\documentclass[11pt]{amsart}
\usepackage[utf8x]{inputenc}
\usepackage{amsfonts}
\usepackage{amssymb,epsfig}
\usepackage{url}
\usepackage{amsmath}
\usepackage{delarray}
\usepackage{graphicx}
\usepackage{fancyhdr}
\usepackage[usenames]{color}

\newtheorem{theorem}{Theorem}
\newtheorem{prop}[theorem]{Proposition}
\newtheorem{lemma}[theorem]{Lemma}
\newtheorem{coro}[theorem]{Corollary}
\newtheorem{defi}[theorem]{Definition}
\newtheorem{rema}[theorem]{Remark}

\newtheorem{notation}[theorem]{Notation}

\def\restr#1#2{\mathchoice
              {\setbox1\hbox{${\displaystyle #1}_{\scriptstyle #2}$}
              \restrictionaux{#1}{#2}}
              {\setbox1\hbox{${\textstyle #1}_{\scriptstyle #2}$}
              \restrictionaux{#1}{#2}}
              {\setbox1\hbox{${\scriptstyle #1}_{\scriptscriptstyle #2}$}
              \restrictionaux{#1}{#2}}
              {\setbox1\hbox{${\scriptscriptstyle #1}_{\scriptscriptstyle #2}$}
              \restrictionaux{#1}{#2}}}
\def\restrictionaux#1#2{{#1\,\smash{\vrule height .8\ht1 depth .85\dp1}}_{\,#2}}

\newcommand{\C}{\ensuremath{\mathbb C}}
\newcommand{\CP}{\mathbb{CP}^1}

\newcommand{\Q}{\ensuremath{\mathbb Q}}
\newcommand{\Z}{\ensuremath{\mathbb Z}}
\newcommand{\N}{\ensuremath{\mathbb N}}
\newcommand{\Diff}{\mathrm{Diff}}

\newcommand{\Sep}{\mathrm{Sep}}

\newcommand{\para}{\mathbin{\!/\mkern-5mu/\!}}

\newcommand{\Id}{\mathrm{Id}}

\newcommand{\m}{\mathfrak{m}}

\newcommand{\DD}{\mathcal{D}}

\newcommand{\FF}{\mathcal{F}}
\newcommand{\tF}{\tilde{\mathcal{F}}}
\newcommand{\RR}{\mathcal{R}}

\newcommand{\tL}{\tilde{\mathcal{L}}}

\newcommand{\MM}{\mathcal{M}}

\newcommand{\LL}{\mathcal{L}}

\newcommand{\tS}{\tilde{S}}

\newcommand{\ddx}{\frac{\partial}{\partial x}}
\newcommand{\ddy}{\frac{\partial}{\partial y}}

\newcommand{\tPhi}{\tilde{\Phi}}

\newcommand{\bx}{\bar{x}}
\newcommand{\by}{\bar{y}}

\begin{document}

\title[Sliding invariants and classification of holomorphic foliations]{Sliding invariants and classification of singular holomorphic foliations in the plane}

\author{{\sc Truong} Hong Minh }
\address{Institut de Mathématiques de Toulouse, UMR5219, Université Toulouse 3}
\email{hong-minh.truong@math.univ-toulouse.fr}

\begin{abstract}{
By introducing a new invariant called the set of slidings, we give a complete strict classification of the class of germs of non-dicritical holomorphic foliations in the plan whose Camacho-Sad indices are not rational. Moreover, we will show that, in this class, the new invariant is finitely determined. Consequently, the finite determination of the class of isoholonomic non-dicritical foliations and absolutely dicritical foliations that have the same Dulac maps are proved.
}\end{abstract}


\maketitle
\section{Introduction}
The problem of classification of germs of foliations in the complex plane is stated by Thom \cite{Zol}. He conjectured that the analytic type of a foliation defined in a neighborhood of a singular point is completely determined by its associated separatrix and its corresponding holonomy. R. Moussu in \cite{Mou} gave a counterexample for this statement and showed that we have to consider the holonomy representation of each irreducible component of  the exceptional divisor in a desingularization instead of the economizes.  For this new version, Thom's problem is proved for cuspidal type singular points \cite{Mou}, \cite{Cer-Mou} and more generally for  quasi-homogeneous foliations \cite{Gen}. However, the new statement of Thom's conjecture was refuted by J.F. Mattei by computing the dimension of the space of isoholonomic deformations \cite{Mat2}, \cite{Mat1}: 
There must be other invariants for the non quasi-homogeneous foliations. This conclusion is confirmed by the number of free coefficients in the normal forms in \cite{Gen-Pau1}, \cite{Gen-Pau2} and in the hamiltonian part of the normal forms of vector field in \cite{Ort-Ros-Vor}.  By adding a new invariant called \emph{set of slidings} this paper solves the problem of strict classification  for the non-dicritical foliations whose Camacho-Sad indices are not rational.  Here, strict classification means up to diffeomorphism tangent to identity.
\subsection{Preliminaries}
A germ of singular foliation $\FF$ in $(\C^2,0)$ is called \emph{reduced} if there exists a coordinate system in which it is defined by a $1$-form whose linear part is 
\begin{equation*}
\lambda_1 y dx + \lambda_2  xdy,\; \frac{\lambda_2}{\lambda_1}\not\in\Q_{< 0},
\end{equation*}
$\lambda=-\frac{\lambda_2}{\lambda_1}$ is called the \emph{Camacho-Sad index} of $\FF$. When $\lambda=0$, the origin is called a \emph{saddle-node} singularity, otherwise it is called \emph{nondegenerate}.  A theorem of A. Seidenberg \cite{Sei, Mat-Mou} says that any singular foliation $\FF$ with isolated singularity admits a canonical \emph{desingularization}. More precisely, there is a holomorphic map
\begin{equation}\label{sigma}
\sigma: \MM\rightarrow (\C^2,0) 
\end{equation}
obtained as a composition of a finite number of blowing-ups at points such that any point $m$ of the \emph{exceptional divisor} $\DD:=\sigma^{-1}(0)$ is  either a regular point or a reduced singularity  of the strict transform $\tilde{\FF}=\sigma^*(\FF)$. An intersection of two irreducible components of $\DD$ is called a \emph{corner}. An irreducible component of $\DD$  is a \emph{dead branch}  if in this component  there is a unique singularity of  $\tF$ that is a corner.

A \emph{separatrix} of $\FF$ is an analytical irreducible invariant curve through the origin of $\FF$. It is well known that  any germ of singular foliation $\FF$ in $(\C^2,0)$ possesses at least one separatrix \cite{Cam-Sad}. When the number of separatrices is finite $\FF$ is \emph{non-dicritical}. Otherwise it is called \emph{dicritical}. 

Denote by $\mathrm{Sing}(\tF)$ the set of all singularities of the strict transform $\tF$. Let $D$ be a non-dicritical irreducible component of the exceptional divisor $\DD$, then $D^*=D\setminus\mathrm{Sing}(\tF)$ is a leaf of $\tF$. Let $m$ be a regular point in $D^*$ and $\Sigma$ a small analytic section through $m$ transverse to $\tF$. For any loop $\gamma$ in $D^*$ based on $m$ there is a germ of a holomorphic return map $ h_\gamma : (\Sigma, m)\rightarrow (\Sigma, m)$ which only depends on the homotopy class of $\gamma$ in the fundamental group $\pi_1(D^*,m)$. The map $h:\pi_1(D^*,m)\rightarrow \mathrm{Diff}(\Sigma,m)$ is called the \emph{vanishing holonomy representation} of $\FF$ on $D$. Let $\FF'$ be a foliation that also admits $\sigma$ as its desingularization map. Assume that $\mathrm{Sing}(\tF')=\mathrm{Sing}(\tF)$ where $\mathrm{Sing}(\tF')$ is the set of singularities of the strict transform $\tF'$. Denote by $h'_\gamma$ in $\mathrm{Diff}(\Sigma,m)$ the vanishing holonomy representation of $\FF$. We say that the vanishing holonomy representation of $\FF$ and $\FF'$ on $D$ are conjugated if there exists $\phi\in\mathrm{Diff}(\Sigma,m)$ such that $\phi\circ h_\gamma= h'_\gamma\circ\phi$. The vanishing holonomy representation of $\FF$ and $\FF'$ are called conjugated if they are conjugated on every non-dicritical irreducible component of $\DD$.
\begin{notation}\label{no1}
We denote by $\mathcal{M}$ the set of all non-dicritical foliations $\FF$ defined on $(\C^2,0)$ such that the Camacho-Sad index of $\tF$ at each singularity is not rational.
\end{notation}
If $\FF$ is in $\mathcal{M}$ then after desingularization all the singularities of $\tF$ are not saddle-node. Moreover, the Chern class of  an irreducible component of divisor, which is an integer, is equal to  the sum of Camacho-Sad indices of the singularities in this component \cite{Cam-Sad}. Therefore, every element in $\MM$ after desingularization admits no dead branch in its exceptional divisor. 

\subsection{Absolutely dicritical foliation}\label{sec1.2}
Let $\sigma$ as in \eqref{sigma} be a composition of a finite number of blowing-ups at points. A germ of singular holomorphic foliation $\LL$ is said $\sigma$-\emph{absolutely dicritical} if the strict transform $\tL=\sigma^*(\LL)$ is a regular foliation and the exceptional divisor $\DD=\sigma^{-1}(0)$ is completely transverse to   $\tL$. When $\sigma$ is the standard blowing-up at the origin, we called $\LL$ a \emph{radial foliation}. At each corner $p=D_i\cap D_j$ of $\DD$, the diffeomorphism from $(D_i,p)$ to $(D_j,p)$ that follows the leaves of $\tL$ is called the \emph{Dulac map} of $\tL$ at $p$. The existence of such foliations for any given $\sigma$ is proved in \cite{Can-Cor}. 
In fact, in \cite{Can-Cor} the authors showed that if in each smooth component of $\DD$ we take any two smooth curves transverse to $\DD$ then there is always an absolutely dicritical foliation admitting  them as their integral curves. We will denote by $\Sep(\tF)\pitchfork\tL$ if at any point $p\in\mathrm{Sing}(\tF)$ the separatrices of $\tF$ through $p$ are transverse to $\tL$. 
\begin{lemma}\label{lem1}
Let $\FF$ be a non-dicritical foliation such that $\sigma$  is its desingularization map. Then there exists a $\sigma$-absolutely dicritical foliation $\LL$ satisfying $\Sep(\tF)\pitchfork\tL$
\end{lemma}
\begin{proof}
Denote by $L_1,\ldots,L_k$ the strict transforms of the separatrices of $\FF$. On each component $D$ of $\DD$ that does not contain any singularity of $\tF$ except the corners we take a smooth curve $L_{k+j}$ transverse to $D$. 
Then we have the set of curves $\{L_1,\ldots,L_n\}$ such that each component of $\DD$ is transverse to at least one curve $L_i$. Denote by $p_i=L_i\cap\DD$. By \cite{Can-Cor}, for each $i$ there exists a $\sigma$-absolutely dicritical foliation $\LL_i$ defined by a $1$-form $\omega_i$ verifying that $L_i$ is transverse to $\tL_i$. Choose a local chart $(x_i,y_i)$ at $p_i$  such that $\DD=\{x_i=0\}, L_i=\{y_i=0\}$ and 
\begin{equation*}
\sigma^*\omega_i(x_i,y_i)=x_i^{m_i}d(x_i+y_i) + h.o.t.,
\end{equation*}  
where ``h.o.t.'' stands for higher order term. Write down $\omega_i$ in the local chart $(x_j,y_j)$
\begin{equation}
\sigma^*\omega_i(x_j,y_j)=x_j^{m_j}d(a_{ij}x_j+b_{ij}y_j) + h.o.t..
\end{equation}
Because $\tL_i$ is transverse to $\DD$, we have $b_{ij}\neq 0$.
We define $a_{ii}=b_{ii}=1$. There always exists a vector $(c_1,\ldots,c_n)\in\C^n$ such that for $j=1,\ldots,n$, 
\begin{equation*}
a_j=\sum_{i=1}^n c_i a_{ij}\neq 0 \;\;\text{and}\;\; b_j=\sum_{i=1}^n c_i b_{ij}\neq 0. 
\end{equation*}
Denote by $\omega_0=\sum c_i\omega_i$. Then, in the local chart $(x_j,y_j)$, we have
\begin{equation*}
\sigma^*\omega_0(x_j,y_j)=x_j^{m_j}d(a_{j}x_j+b_{j}y_j) + h.o.t.,\;\;\text{where}\;\; a_j\ne 0, b_j\ne 0.
\end{equation*}
Because $\omega_0$ and $\omega_i$, $i=1,\ldots, n$, have the same multiplicity on each component of $\DD$, they have the same vanishing order. Since each component of $\DD$ contains at least one point $p_i$ and the strict transform $\tL$ of the foliation $\LL$ defined by $\omega_0$ is transverse to $\DD$ in each neighborhood of each $p_i$, $\tL$ is generically transverse to $\DD$. By \cite{Can-Cor}, $\LL$ is absolutely dicritical and satisfies $\Sep(\tF)\pitchfork\tL$.
\end{proof}

\subsection{Slidings of foliations}
Consider first a nondegenerate reduced foliation $\FF$ in $(\C^2,0)$. By \cite{Mat-Mou}, there exists a coordinate system in which $\FF$ is defined by 
$$\lambda y(1+A(x,y))dx + xdy,\;\lambda\notin\Q_{\leq 0},$$
where $A(0,0)=0$. Let $\LL$ be a germ regular foliation whose invariant curve through the origin (we call it the separatrix of $\LL$) is transverse to the two separatrices of $\FF$, which are denoted by $S_1$ and $S_2$. Then we have the following Lemma, whose proof is straightforward. 
\begin{lemma}\label{lem2}
The tangent curve of $\FF$ and $\LL$, denoted  $T(\FF,\LL)$, is a smooth curve transverse to the two separatrices of $\FF$. Moreover, if the separatrix of $\LL$ is tangent to $\{x-cy=0\}$ then $T(\FF,\LL)$ is tangent to $\{x+c\lambda y=0\}$.
\end{lemma}
\begin{figure}
\begin{center}
\includegraphics[scale=.45]{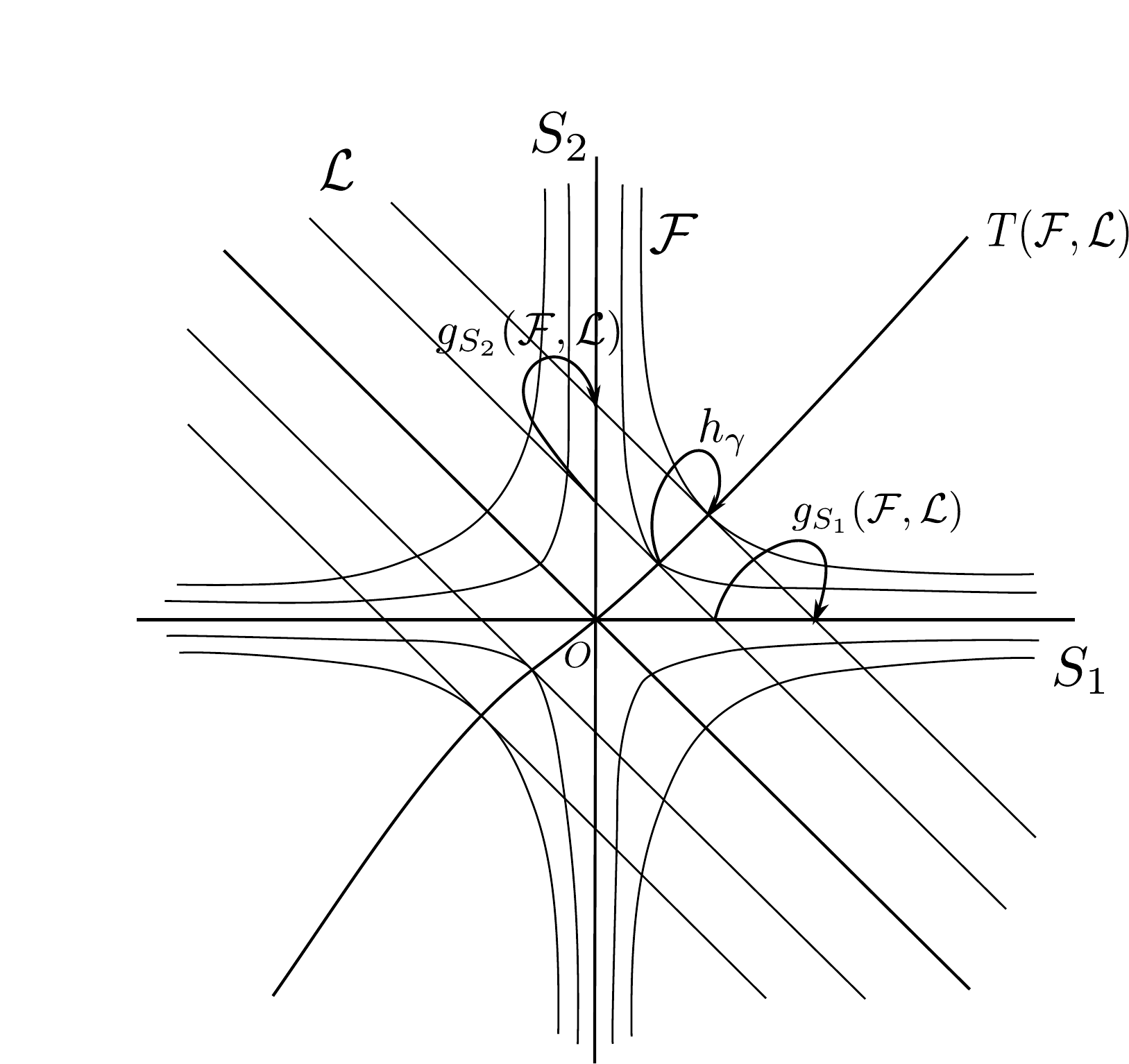}
\caption{Sliding of $\FF$ and $\LL$. }
\label{figure1}
\end{center}
\end{figure}
After a standard blowing-up $\sigma_1$ at the origin, the strict transform $\tilde{T}(\FF,\LL)$ of $T(\FF,\LL)$ is transverse to $\tF$ and cut $D_1=\sigma_1^{-1}(0)$ at $p$. We denote by $D_1^*=D_1\setminus \mathrm{Sing}(\sigma_1^*(\FF))$ and  $\tilde{h}:\pi_1(D_1^*,p)\rightarrow \mathrm{Diff}(\tilde{T}(\FF,\LL),p)$ the vanishing holonomy representation of $\FF$. We choose a generator $\gamma$ for $\pi_1(D_1^*,p)\cong \Z$. Then $\sigma_1$ induces  
$$h_\gamma=\sigma_1\circ \tilde{h}(\gamma)\circ\sigma_1^{-1}\in\Diff(T(\FF,\LL),0).$$
We call $h_\gamma$ the \emph{holonomy on the tangent curve $T(\FF,\LL)$}. Denote by $\pi_{S_1}$ and $\pi_{S_2}$ the projection by the leaves of $\LL$ from $T(\FF,\LL)$ to $S_1$ and $S_2$ respectively.
\begin{defi}
The sliding of a reduced foliation $\FF$ and a regular foliation $\LL$ on $S_1$ (resp., $S_2$) is the diffeomorphism (figure \ref{figure1})
\begin{align*}
g_{S_1}(\FF,\LL)&= \pi_{S_1*}(h_\gamma)=\pi_{S_1}\circ h_\gamma\circ\pi_{S_1}^{-1}\\
\big(\text{resp.,}\;\;g_{S_2}(\FF,\LL)&= \pi_{S_2*}(h_\gamma)=\pi_{S_2}\circ h_\gamma\circ\pi_{S_2}^{-1}\big).
\end{align*}\end{defi}
Let $d:S_1\rightarrow S_2$ be the Dulac map of $\LL$ (Section \ref{sec1.2}). Since $d=\pi_{S_2}\circ\pi_{S_1}^{-1}$, it is obvious that
\begin{equation}\label{3}
g_{S_2}(\FF,\LL)=d_*\left(g_{S_1}(\FF,\LL)\right).
\end{equation}

Now let $\FF$ be a non-dicritical foliation such that after desingularization by the map $\sigma$ all singularities of $\sigma^*(\FF)=\tF$ are nondegenerate. By Lemma \ref{lem1} there exists a $\sigma$-absolutely dicritical foliation $\LL_0$ such that $\Sep(\tF)\pitchfork\tL_0$. 
\begin{notation}\label{no5}
We denote by $\RR_\FF(\LL_0)$ the set of all $\sigma$-absolutely dicritical foliations $\LL$ satisfying the two following properties:
\begin{itemize}
\item $\tL$ and $\tL_0$ have the same Dulac maps at any corner of $\DD$.
\item At each singularity $p$ of $\tF$, the invariant curves of $\tL$ and $\tL_0$ through $p$ are tangent (figure \ref{p38}).  
\end{itemize} 
\end{notation}

\begin{figure}
\begin{center}
\includegraphics[scale=0.8]{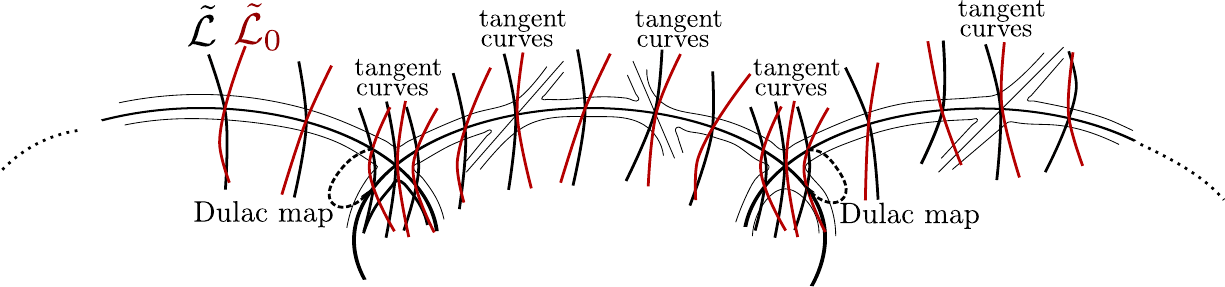}
\caption{Element $\LL$ of $\mathcal{R}_{\FF}(\LL_0)$}\label{p38}
\end{center}
\end{figure}

Let $\LL$ in $\RR_\FF(\LL_0)$ and $D$ be an irreducible component of $\DD$. Suppose that $p_1,\ldots,p_m$ are the singularities of $\tF$ on $D$. Then we denote by 
$$S_D(\tF,\tL)=\{g_{D,p_1}(\tF,\tL),\ldots,g_{D,p_m}(\tF,\tL)\},$$
where $g_{D,p_i}(\tF,\tL)$ is the sliding of $\tF$ and $\tL$ in a neighborhood of $p_i$.
\begin{defi}
The sliding of $\FF$ and $\LL$ is 
\begin{equation*}
S(\FF,\LL)=\cup_{D\in\mathrm{Comp}(\DD)}S_D(\tF,\tL),
\end{equation*}
where $\mathrm{Comp}(\DD)$ is the set of all irreducible components of $\DD$. The set of slidings of $\FF$ relative to direction $\LL_0$ is the set
\begin{equation*}
\mathfrak{S}(\FF,\LL_0)=\cup_{\LL\in\RR_\FF(\LL_0)}S(\FF,\LL).
\end{equation*} 
\end{defi}

We will prove in Corollary \ref{cor3} that $\mathfrak{S}(\FF,\LL_0)$ is an invariant of $\FF$: If $\FF$ and $\FF'$ are conjugated by $\Phi$ then for each $\LL$ in $\RR_\FF(\LL_0)$ we have $S(\FF,\LL)=\tPhi_{|\DD}\circ S(\FF',\Phi_*\LL)\circ\tPhi^{-1}_{|\DD}$ . Under some conditions for $\FF$ and $\FF'$ (Theorem \ref{thr1}), we will have $\tPhi_{|\DD}=\Id$. Therefore $S(\FF,\LL)=S(\FF',\Phi_*\LL)$. Moreover, $\Phi_*\LL$ is also in $\RR_\FF(\LL_0)$. Consequently,  $\mathfrak{S}(\FF,\LL_0)=\mathfrak{S}(\FF',\LL_0)$.

\begin{rema}\label{re7} For each singularity $p$ of $\tF$ that is a corner, i.e., $p=D_i\cap D_j$, there are two slidings $g_{D_i,p}(\tF,\tL)$ and $g_{D_j,p}(\tF,\tL)$. However, by \eqref{3}, $g_{D_j,p}(\tF,\tL)$ is completely determined by $g_{D_i,p}(\tF,\tL)$ and the Dulac map of $\tL$ at $p$.\\

\noindent This invariant is named ``sliding'' because it gives an obstruction for the construction of local conjugacy of two foliations that fixes the points in the exceptional divisor (Corollary \ref{cor3}).\\

\noindent The definition of $\mathfrak{S}(\FF,\LL_0)$ does not depend on choosing a element $\LL_0$ in $\RR_\FF(\LL_0)$. More precisely, if $\LL'_0\in\RR_\FF(\LL_0)$ then  $\LL_0\in\RR_\FF(\LL'_0)$ and $\mathfrak{S}(\FF,\LL_0)=\mathfrak{S}(\FF,\LL'_0)$\\

\noindent Although $S(\FF,\LL)$ is a set of local diffeomorphisms, it is not a local invariant. $S(\FF,\LL)$ also contains the information of the relation of those local diffeomorphisms because  all these local diffeomorphisms are defined by the holonomy projections following the global fibration $\LL$: in some sense, any fibration $\LL\in\RR_\FF(\LL_0)$ plays the role of a global common transversal coordinate on which the slidings invariants are computed all together and at the same time.

\end{rema}

Let us clarify here the role of the sliding invariant in the problem of classification
of germs of foliations. Suppose that two non-dicritical foliations $\FF$ and $\FF'$ satisfy that their separatrices and their vanishing holonomies are conjugated. Moreover, after desingularization, all the Camacho-Sad indices of $\FF$ and $\FF'$ are coincide. Then after blowing-ups, $\FF$ and $\FF'$ are locally conjugated in a neighborhood of their singularities. Although we have the conjugation of their vanishing holonomies, in general, we can not glue the local conjugation together. The obstruction is that these local conjugations induce the local diffeomorphisms on the exceptional divisor which we call the slidings. In general, there is no reason for those slidings  being parts of a global diffeomorphism of the divisor. 

\subsection{Statement of the main results}
Let $\FF, \FF'\in\mathcal{M}$. We say that their \emph{strict separatrices are tangent}, denoted $\Sep(\tF)\para \Sep(\tF')$, if they have the same desingularization map and the same set of singularities. Moreover, at each singularity which is not a corner of the divisor the separatrices of $\tF$ and $\tF'$ are tangent. If $\Sep(\tF)\para \Sep(\tF')$ and $\LL_0$ is an absolutely dicritical foliation satisfying $\Sep(\tF)\pitchfork\tL_0$ then $\Sep(\tF')\pitchfork\tL_0$ and $\RR_\FF(\LL_0)=\RR_{\FF'}(\LL_0)$. We denote by $\mathrm{CS}(\tF)$ the set of Camacho-Sad indices of $\tF$ at all singularities. We also denote by $\mathrm{CS}(\tF)=\mathrm{CS}(\tF')$  if at each singularity, $\tF$ and $\tF'$ have the same Camacho-Sad index.   

\begin{theorem}\label{thr1}
Let $\FF$ and $\FF'$ be two foliations in the class $\mathcal{M}$ (see Notation \ref{no1}) such that $\Sep(\tF)\para\Sep(\tF')$. Suppose that $\LL_0$ is an absolutely dicritical foliation satisfying $\Sep(\tF)\pitchfork\tL_0$. Let $\RR_\FF(\LL_0)$ be as in Notation \ref{no5} and $\mathfrak{S}(\FF,\LL_0)$, $\mathfrak{S}(\FF',\LL_0)$  the corresponding sets of slidings. Then the three following statements are equivalent:
\begin{enumerate}
\item[(i)] $\FF$ and $\FF'$ are strictly analytically conjugated.
\item[(ii)] Their vanishing holonomy representations are strictly analytically conjugated, $\mathrm{CS}(\tF)=\mathrm{CS}(\tF')$ and $\mathfrak{S}(\FF,\LL_0)=\mathfrak{S}(\FF',\LL_0)$.
\item[(iii)] Their vanishing holonomy representations are strictly analytically conjugated, $\mathrm{CS}(\tF)=\mathrm{CS}(\tF')$ and $\mathfrak{S}(\FF,\LL_0)\cap\mathfrak{S}(\FF',\LL_0)\neq\emptyset$.  
\end{enumerate}
\end{theorem}

Here, a strict conjugacy means a conjugacy tangent to identity. We will prove that  the slidings of foliations are finitely determined:
\begin{theorem}\label{thr2}
Let $\FF$ be a non-dicritical foliation without saddle-node singularities after desingularization. There exists a natural $N$ such that if there is a non-dicritical foliation $\FF'$ satisfying the following conditions: 
\begin{enumerate}
\item[(i)] $\FF$ and $\FF'$ have the same set of singularities after desingularization and at a neighborhood of each singularity, $\tF$ and $\tF'$ are locally strictly analytically conjugated,
\item[(ii)] There exist $\LL,\LL'$ in $\RR_\FF(\LL_0)$ such that $J^N(S(\FF,\LL))=J^N(S(\FF',\LL'))$,
\end{enumerate}
then there exists $\LL''$ such that $\LL''$ is strictly conjugated with $\LL$ and $S(\FF,\LL'')=S(\FF',\LL')$.
\end{theorem}
Here $J^N(S(\FF,\LL))=J^N(S(\FF',\LL'))$ means $J^{N}(g_{D,p}(\tF,\tL))=J^{N}(g_{D,p}(\tF',\tL'))$ for all $g_{D,p}(\tF,\tL)$ in $S(\FF,\LL)$, $g_{D,p}(\tF',\tL')$ in $S(\FF',\LL')$, where $J^N(g_{D,p}(\tF,\tL))$ stands for the regular part of degree $N$ in the Taylor expansion of $g_{D,p}(\tF,\tL)$. \\  

These two theorems also give two corollaries on finite determination of the class of isoholonomic non-dicritical foliations and absolutely dicritical foliations that have the same Dulac maps (see Corollary \ref{cor17} and \ref{cor19}). \\  
 
This paper is organized as follows: In section 2, local conjugacy of the pair $(\FF,\LL)$  will be proved. We prove Theorem \ref{thr1} in Section 3. Section 4 is devoted to prove Theorem \ref{thr2} and two Corollaries of finite determination of class of isoholonomic non-dicritical foliations and absolutely dicritical foliations that have the same Dulac maps.


\section{Local conjugacy of the pair $(\FF,\LL)$}
Let $\FF$, $\FF'$ be two germs of nondegenerate reduced foliations in $(\C^2,0)$. Denote by $S_1$, $S_2$ and $S'_1$, $S'_2$ the separatrices of $\FF$ and $\FF'$ respectively. Let $\LL$ and $\LL'$ be two germs of regular foliations such that their separatrices $L$ and $L'$ are transverse to the two separatrices of  $\FF$ and $\FF'$ respectively. Suppose that $\Phi$ is a diffeomorphism conjugating $(\FF,\LL)$ and $(\FF',\LL')$, then the restriction of $\Phi$ on the tangent curves commutes with the holonomies on $T(\FF,\LL)$ and $T(\FF',\LL')$ of $\FF$ and $\FF'$. The converse is also true: 

\begin{prop}\label{pro8}
Suppose that $\FF$ and $\FF'$ have the same Camacho-Sad index. If $\phi:T(\FF,\LL)\rightarrow T(\FF',\LL')$ is a diffeomorphism commuting with the holonomies of $\FF$ and $\FF'$ then $\phi$ extends to a diffeomorphism $\Phi$ of $(\C^2,0)$ sending $(\FF,\LL)$ to $(\FF',\LL')$.  Moreover,  if we require that $\Phi$ sends $S_1$ (resp. $S_2$) to $S'_1$ (resp. $S'_2$) then this extension is unique.
\end{prop}
\begin{proof}
By Lemma \ref{lem1}, the curves $S_1$, $S_2$, $L$, $T(\FF,\LL)$ (resp., $S'_1$, $S'_2$, $L'$, $T(\FF',\LL')$) are four transverse smooth curves. It is well known that there exist two radial foliations $\RR$ and $\RR'$ such that $S_1$, $S_2$, $L$, $T(\FF,\LL)$ and $S'_1$, $S'_2$, $L'$, $T(\FF',\LL')$ are the invariant curves of $\RR$ and $\RR'$ respectively. After a blowing-up at the origin, denote by $p_1$, $p_2$, $p_L$, $p_T$ (resp., $p'_1$, $p'_2$, $p'_L$, $p'_T$) the intersections of strict transforms of $S_1$, $S_2$, $L$, $T(\FF,\LL)$ (resp., $S'_1$, $S'_2$, $L'$, $T(\FF',\LL')$) with $\CP$. Take $\phi_1$ in $\mathrm{Aut}(\CP)$ that sends $p_1$, $p_2$, $p_L$ to $p'_1$, $p'_2$, $p'_L$ respectively. By Lemma \ref{lem1} , the direction of $T(\FF,\LL)$ (resp., $T(\FF',\LL')$) is completely determined by the Camacho-Sad index and the direction of $L$ (resp., $L'$). Therefore, $\phi_1(p_T)=p'_T$. Using the path lifting method after a blowing-up \cite{Mat-Mou}, $\phi$ extends to a diffeomorphism $\Phi_1$ of $(\C^2,0)$ sending $(\FF,\RR)$ to $(\FF',\RR')$.

Denote by $\LL_0=\Phi_{1*}^{-1}(\LL')$. Because 
$\Phi_1^{-1}$ sends $L'$ and $T(\FF',\LL')$ to $L$ and $T(\FF,\LL)$ respectively, $L$ is also the separatrix of  $\LL_0$ and $T(\FF,\LL_0)=T(\FF,\LL)$. We denote by $T$ the tangent curve  $T(\FF,\LL_0)$. The proof is reduced to show that there exists a diffeomorphism fixing points in $T$ sending $(\FF,\LL)$ to $(\FF,\LL_0)$. Choose a system of coordinates $(x,y)$ such that $\LL_0$ is defined by $f_0=x+y$ and $\FF$ is defined by a $1$-form
$$\omega(x,y)=\lambda y(1+A(x,y))dx+xdy,\; \lambda\not\in\Q_{\leq 0}.$$
Then $T$ is defined by 
$$\tau(x,y)=x-\lambda y(1+A(x,y))=0.$$
We claim that there exist a natural $n\geq 2$ and a holomorphic function $h$ such that $\LL$ is defined by
$$f(x,y)=\left(1+\tau^n(x,y)h(x,y)\right)(x+y).$$
Indeed, assume that $\LL$ is defined by 
$$\bar{f}(x,y)=u(x,y)(x+y),$$
where $u$ is invertible. Rewrite the equation of $T$ as 
$$x-\bar\tau(y)=0,$$  
where $\bar\tau(y)=\lambda y + \ldots$. Because $\lambda\neq -1$, the maps $u(\bar\tau(y),y).(\bar\tau(y)+y)$ and $\bar\tau(y)+y$ are  diffeomorphic. Hence there exists a diffeomorphism $g\in\C\{y\}$  such that
$$g\Bigl(u(\bar\tau(y),y).(\bar\tau(y)+y)\Bigr)=\bar\tau(y)+y.$$
This is equivalent to
\begin{equation*}
\bigl(g\circ \bar{f}-(x+y)\bigr)_{| \tau=0}=0.
\end{equation*}
Therefore, there exist a natural $n\geq 1$ and a function $h$ satisfying $h_{|\tau=0}\not\equiv 0$ such that
$$g\circ \bar{f}(x,y)=\left(1+\tau^n(x,y)h(x,y)\right)(x+y).$$
Because $g$ is a diffeomorphism, $\LL$ is also defined by $f=g\circ \bar{f}$. 

Let us prove $n\geq 2$. We have
\begin{align*}
df\wedge \omega&=\tau(x,y)(\ldots)+ n(x+y)h(x,y)\tau^{n-1}d\tau\wedge\omega  \\
& =\tau(x,y)(\ldots)+ n(x+y)h(x,y)\tau^{n-1}(x+\lambda^2y+\ldots)dx\wedge dy.
\end{align*}
Because $T$ is defined by $\tau(x,y)=0$, we have
\begin{equation*}
 n(x+y)h(x,y)\tau^{n-1}(x+\lambda^2y+\ldots)_{|\tau=0}\equiv 0
\end{equation*}
The fact $\lambda\neq 0,-1$ forces to $x+\lambda^2y\neq x-\lambda y$ and $x+y\neq x-\lambda y$. This implies 
$(\tau^{n-1})_{|\tau=0}\equiv 0$. Consequently, $n\geq 2$.

Now let
$$X=x\ddx- \lambda y(1+A(x,y)) \ddy$$
tangent to $\FF$. Now we will show that there exists $\alpha\in\C\{x,y\}$ such that the diffeomorphism 
$\exp[\tau^{n-1}\alpha]X$ satisfies 
\begin{equation}\label{eq1}
(x+y)\circ\exp[\tau^{n-1}\alpha]X(x,y)=\sum_{i\geq 0}\frac{\tau^{i(n-1)}\alpha^i}{i!}\mathrm{ad}^i_{X}(x+y)=f(x,y),
\end{equation}
where $\mathrm{ad_X}$ is the adjoint representation. Since $$\sum_{i\geq 0}\frac{\tau^{i(n-1)}\alpha^i}{i!}\mathrm{ad}^i_{X}(x+y)=x+y+\tau^n\alpha + \frac{n}{2}\tau^{2n-2}\alpha^2X(\tau)+\tau^{2n-1}(\ldots),$$
\eqref{eq1} becomes
$$\alpha + \frac{n}{2}\tau^{n-2}\alpha^2X(\tau)+\tau^{n-1}(\ldots)=(x+y)h(x,y).$$
Hence, the existence of $\alpha$ comes from the implicit function theorem.\\

Now we will prove the uniqueness of $\Phi$. In fact, we only need to show that if there exists a diffeomorphism $\Psi$ that sends $(\FF,\LL_0)$ to itself, preserves the two separatrices of $\FF$ and fixes the points of $T$ then $\Psi=\mathrm{Id}$. Since $\Psi_{|T}=\mathrm{Id}$, $\Psi$ sends every leaf of $\FF$ into itself. By \cite{Ber-Cer-Mez}, there exists $\beta\in\C\{x,y\}$ such that
$$\Psi=\exp[\beta]X.$$ 
Because $\LL_0$ is defined by the function $x+y$ and $\Psi$ fixes points in $T$, we get
\begin{equation}\label{eq2}
(x+y)\circ\exp[\beta]X=x+y.
\end{equation}
Decompose $\beta$ into the homogeneous terms
$$\beta=\beta_0+\beta_1+\beta_2+\ldots=\beta_0+\bar{\beta}.$$
Since $\mathrm{ad}^i_X(x)=x$ and $\mathrm{ad}^i_X(y)=(-\lambda)^iy+c_i$ for all $i$, where $c_i\in(x,y)^2$, we have
\begin{align*}(x+y)\circ\exp[\beta]X&=\sum_{i=0}^{\infty}\frac{\beta_0^i}{i!}x+\sum_{i=0}^{\infty}\frac{\beta_0^i}{i!}((-\lambda)^iy)+h.o.t.\\
&=\exp(\beta_0)x+\exp(-\lambda\beta_0)y+h.o.t..
\end{align*}
So \eqref{eq2} leads to
$$\exp(\beta_0)=\exp(-\lambda\beta_0)=1.$$
Hence, 
\begin{align}
x\circ\exp[\beta_0]X&=\sum_{i=0}^{\infty}\frac{\beta_0^i}{i!}x=\exp(\beta_0) x=x, \label{5}\\
y\circ\exp[\beta_0]X&=\sum_{i=0}^{\infty}\frac{\beta_0^i}{i!}\left((-\lambda)^i y+c_i\right)=\exp(-\lambda \beta_0) y+c=y+c,\label{6}
\end{align}
where $c\in(x,y)^2$. We claim that 
\begin{equation}\label{7}
\exp[\beta]X=\exp[\beta_0]X\circ\exp[\bar\beta]X.
\end{equation}
Indeed, for any $h\in\C\{x,y\}$ we have
\begin{align*}
h\circ\exp[\beta_0]X\circ\exp[\bar\beta]X &=\left(\sum_{i=0}^{\infty}\frac{\beta_0^i}{i!}\mathrm{ad}^i_X(h)\right)\circ\exp[\bar\beta]X
=\sum_{j=0}^{\infty}\frac{\bar\beta^j}{j!}\mathrm{ad}^j_X\left(\sum_{i=0}^{\infty}\frac{\beta_0^i}{i!}\mathrm{ad}^i_X(h)\right)\\
&=\sum_{k=0}^{\infty}\sum_{i+j=k}\frac{\bar\beta^j \beta_0^i}{j!i!}\mathrm{ad}^k_X(h)=\sum_{k=0}^{\infty} \frac{(\bar\beta+\beta_0)^k}{k!}\mathrm{ad}^k_X(h)=h\circ\exp[\beta]X.
\end{align*}
We write $\mathrm{ad}^i_X(y+c)=(-\lambda)^iy+d_i$ where $d_i\in(x,y)^2$. By \eqref{5}, \eqref{6}, \eqref{7} we get
\begin{align}
(x+y)\circ\exp[\beta]X&=x\circ\exp[\beta_0]X\circ \exp[\bar\beta]X+ y\circ\exp[\beta_0]X\circ\exp[\bar\beta]X\nonumber\\
&=x\circ\exp[\bar\beta]X+ (y+c)\circ\exp[\bar\beta]X\nonumber\\
&=\sum_{i=0}^{\infty}\frac{\bar\beta^i}{i!}x+\sum_{i=0}^{\infty}\frac{\bar\beta^i}{i!}\mathrm{ad}^i_X(y+c)\nonumber\\
&= \exp(\bar\beta)x+\exp(-\lambda\bar{\beta})y+ \sum_{i=0}^{\infty}\frac{\bar\beta^i}{i!}d_i\nonumber\\
&=x\prod_{i=1}^{\infty}\exp(\beta_i)+ y\prod_{i=1}^{\infty}\exp(-\lambda\beta_i)+ \sum_{i=0}^{\infty}\frac{\bar\beta^i}{i!}d_i.\label{8}
\end{align}
We will prove $\bar{\beta}=0$ by induction. From \eqref{8} we have
$$(x+y)\circ\exp[\beta]X=x(1+\beta_1)+ y(1-\lambda\beta_1)+h.o.t.$$
So \eqref{eq2} forces  $\beta_1=0$. Suppose that $\beta_1=\ldots=\beta_{k-1}=0$, we have
$$(x+y)\circ\exp[\beta]X=x(1+\beta_k)+ y(1-\lambda\beta_k)+h.o.t..$$
Then \eqref{eq2} again leads to $\beta_k=0$ and consequently $\beta=\beta_0$. This implies that
$$\Psi=\exp[\beta_0]X=(x,y+c).$$
Finally,  \eqref{eq2} again gives $c=0$. So $\Psi=\mathrm{Id}$.
\end{proof}
\begin{coro}\label{cor11}
Suppose that $\FF$ and $\FF'$ are two nondegenerate reduced foliations that are analytically conjugated. Let $\LL$ and $\LL'$ be two regular foliations that are transverse to the two separatrices of $\FF$ and $\FF'$ respectively. Then there exists a diffeomorphism that sends $(\FF,\LL)$ to $(\FF',\LL')$. 
\end{coro}
\begin{proof}
Let $\Psi$ be the conjugacy of $\FF$ and $\FF'$. Denote by $T'=\Psi(T(\FF,\LL))$. Then the restriction $\Psi_{|T(\FF,\LL)}$ commutes with the holonomies of $\FF$ on $T(\FF,\LL)$ and $\FF'$ on $T'$. Moreover by the holonomy transport, the holonomies of $\FF'$ on $T'$ and on $T(\FF',\LL')$ are conjugated. Hence, the holonomies of $\FF$ on $T(\FF,\LL)$ and $\FF'$ on $T(\FF',\LL')$ are conjugated. By Proposition \ref{pro8} there exists a diffeomorphism that sends $(\FF,\LL)$ to $(\FF',\LL')$
\end{proof}

By projecting  on $S_1$ and $S_2$ the  holonomies  defined on $T(\FF,\LL)$ and $T(\FF',\LL')$ respectively, we can obtain
\begin{coro}\label{cor3}
If $\Phi$ is a diffeomorphism conjugating $(\FF,\LL)$ and $(\FF',\LL')$, then 
$$\Phi_{|S_1}\circ g_{S_1}(\FF,\LL)=g_{S'_1}(\FF',\LL')\circ\Phi_{|S_1}.$$
Reciprocally, if $\mathrm{CS}(\FF)=\mathrm{CS}(\FF')$ and $\phi:S_1\rightarrow S'_1$ is a diffeomorphism satisfying
$$\phi\circ g_{S_1}(\FF,\LL)=g_{S'_1}(\FF',\LL')\circ\phi$$
then $\phi$  uniquely extends to a diffeomorphism $\Phi$ of $(\C^2,0)$ sending $(\FF,\LL)$ to $(\FF',\LL')$.
\end{coro}
\begin{proof}
Because $\Phi$ conjugates $(\FF,\LL)$ and $(\FF',\LL')$, the restriction $\Phi_{|T(\FF,\LL)}$ commutes with the holonomies $h_\gamma$ and $h'_\gamma$ of $\FF$ and $\FF'$. Denote by $\pi_{S_1}$ (resp., $\pi_{S'_1}$) the projection by the leaves of $\LL$ (resp., $\LL'$) from $T(\FF,\LL)$ (resp., $T(\FF',\LL')$) to $S_1$ (resp., $S'_1$). Since $\Phi$ sends $(\FF,\LL)$ to $(\FF',\LL')$, we have
\begin{equation*}
\pi_{S'_1}\circ\Phi_{|T(\FF,\LL)}=\Phi_{|S_1}\circ\pi_{S_1}.
\end{equation*} 
Therefore
\begin{align*}
\Phi_{|S_1}\circ g_{S_1}(\FF,\LL)&=\Phi_{|S_1}\circ\pi_{S_1}\circ h_\gamma\circ\pi_{S_1}^{-1}=\pi_{S'_1}\circ\Phi_{|T(\FF,\LL)}\circ h_\gamma\circ\pi_{S_1}^{-1}\\
&=\pi_{S'_1}\circ h'_\gamma\circ\Phi_{|T(\FF,\LL)}\circ\pi_{S_1}^{-1}=g_{S'_1}(\FF',\LL')\circ\pi_{S'_1}\circ\Phi_{|T(\FF,\LL)}\circ\pi_{S_1}^{-1}\\
&=g_{S'_1}(\FF',\LL')\circ\Phi_{|S_1}.
\end{align*}

Reciprocally,  suppose $\phi:S_1\rightarrow S'_1$ is a diffeomorphism commuting with the slidings of $\FF$ and $\FF'$. Denote by $\psi=\pi^{-1}_{S'_1} \circ\phi\circ\pi_{S_1}$ then
\begin{align*}
\psi\circ h_\gamma&= \pi^{-1}_{S'_1} \circ\phi\circ\pi_{S_1}\circ h_\gamma= \pi^{-1}_{S'_1} \circ\phi\circ g_{S_1}(\FF,\LL)\circ\pi_{S_1}\\&= \pi^{-1}_{S'_1} \circ g_{S'_1}(\FF',\LL')\circ\phi\circ\pi_{S_1}= h'_\gamma\circ \pi^{-1}_{S'_1} \circ\phi\circ\pi_{S_1}=h'_\gamma\circ\psi.
\end{align*}
By Proposition \ref{pro8}, $\psi$ uniquely extends to a diffeomorphism $\Phi$ that sends $(\FF,\LL)$ to $(\FF',\LL')$. 
\end{proof}

\begin{rema}
In particular, if in Corollary \ref{cor3} we have  $S_1=S'_1$ and $g_{S_1}(\FF,\LL)=g_{S'_1}(\FF',\LL')$ then there exists a diffeomorphism sending $(\FF,\LL)$ to $(\FF',\LL')$ and fixing points in $S_1$.  
\end{rema}


\section{Strict classification of foliations in $\mathcal{M}$}\label{sec3}
This whole section is devoted to prove Theorem \ref{thr1}.

\begin{proof}[Proof of theorem \ref{thr1}]
The direction ((ii)$\Rightarrow$(iii)) is obvious.\\

((i)$\Rightarrow$(ii)) Since the Camacho-sad index is an analytic invariant, it is obvious that $\mathrm{CS}(\tF)=\mathrm{CS}(\tF')$.

Let $\Phi$ be the strict conjugacy and $\tilde{\Phi}:(\MM,\DD)\rightarrow (\MM,\DD)$ be its lifting by $\sigma$. Suppose that  a non-corner point $m$ of $\DD$ is a fixed point of $\tilde{\Phi}$. Then the linear map $D\tilde{\Phi}(m)$ has two eigenvalues. One corresponds to the direction of the divisor. We denote by $v(\tPhi)(m)$ the other eigenvalue and define $v(\tPhi)(m)=1$ for each corner $m$. 
\begin{lemma}\label{lem9}
$\tPhi_{|\DD}=\Id$ so $v(\tPhi)$ is a function defined on $\DD$ and moreover  $v(\tPhi)\equiv 1$.
\end{lemma}
\begin{proof}
Denote by $\sigma_1$ the standard blowing-up at the origin of $(\C^2,0)$
$$\sigma_1:(\MM_1,D_1)\rightarrow(\C^2,0).$$
On $D_1$, we use the two standard chart $(x,\by)$ and $(\bx, y)$ together with the transition functions $\bx=\by^{-1}$, $y=x\by$. 
Suppose that 
$$\Phi(x,y)=(x+\alpha(x,y), y+\beta(x,y)),\;\alpha,\beta\in(x,y)^2.$$
Then in the coordinate system $(x,\by)$ we have
\begin{align*}\Phi_1(x,\by)=\sigma_1^*\Phi(x,\by)
&=\left(x+\alpha(x,x\by),\frac{x\by+\beta(x,x\by)}{x+\alpha(x,x\by)}\right)\\&=(x(1+\ldots), \by+x(\beta_0+\ldots)),
\end{align*}
where $\beta_0=\frac{\partial^2\beta}{\partial x^2}(0,0)$. Therefore $\Phi_1:(\MM_1,\DD_1)\rightarrow(\MM_1,D_1)$ fixes points in $D_1$ and $v(\Phi_1)\equiv 1$. Let $p$ be a non-reduced singularity of $\sigma_1^*\FF$ on $D_1$. We will show that $D\Phi_1(p)=\Id$ and apply the inductive hypothesis for $\Phi_1$ in a neighborhood of $p$.  Indeed, let $\sigma_2$ be the blowing-up at $p$ and $D_2=\sigma_2^{-1}(p)$. Denote by $S_p$ and $S'_p$ all invariant curves of $\sigma_1^*\FF$ and $\sigma_1^*\FF'$ through $p$. Because every element in $\MM$ after desingularization admits no dead component in its exceptional divisor, $D_2$ is not a dead component. Therefore there is at least one irreducible component $\ell_p$ of $S_p$ that are not tangent to $D_1$. Because $\Phi_{1*}(S_p)=S'_p$ and $\Sep(\tF)\para\Sep(\tF')$,  $D\phi_1(p)$ has an eigenvector different from the direction of $D_1$, which is corresponding to the direction of $\ell_p$. So $D\Phi_1(p)$ has two eigenvectors. Since both of their eigenvalues are $1$, we have $D\Phi_1(p)=\Id$. 
\end{proof}

Now let $\LL\in\mathcal{R}_{\FF}(\LL_0)$ and denote by $\LL'=\Phi_*(\LL)$. Since $\tPhi_{|\DD}=\Id$, the strict transforms $\tL$ and $\tL'$ have the same Dulac maps. Moreover, because $\Sep(\tF)\para\Sep(\tF')$, at each singularity $p$ of $\tF$,  $D\tPhi(p)$ has two eigenvectors. As $v(\tPhi)\equiv 1$ and $\tPhi_{|\DD}=\Id$ we have $D\tPhi(p)=\mathrm{Id}$. Therefore the invariant curves of $\tL$ and $\tL'$ through $p$ are tangent. This gives $\LL'=\Phi_*(\LL)\in\RR_\FF(\LL_0)$. Because $\tPhi$ fixes points in $\DD$, by Corollary \ref{cor3} the identity map commutes with the slides of $\FF$ and $\FF'$. This leads to $S(\FF,\LL)=S(\FF',\LL')$. Consequently, $\mathfrak{S}(\FF,\LL_0)=\mathfrak{S}(\FF',\LL_0)$. Moreover, the vanishing holonomy representation of $\FF$ and $\FF'$ are conjugated by $\tPhi$. Since $v(\tPhi)\equiv 1$ this conjugacy is strict.\\

((iii)$\Rightarrow$(i))
Suppose that $\LL$, $\LL'\in\mathcal{R}_{\FF}(\LL_0)$ satisfy $S(\FF,\LL)=S(\FF',\LL')$. By Corollary \ref{cor3}, at each singularity $p_i$, $i\in\{1,\ldots,k\}$, of $\tF$ there exists a neighborhood $U_i$ of $p_i$ and a local conjugacy
$$\Phi_i:(\tF,\tL)_{|U_i}\rightarrow (\tF',\tL')_{|U_i}$$
such that $\Phi_{i|\DD\cap U_i}=\Id$.
Let $U_0$ be a neighborhood of $\DD\setminus\cup_{i=1}^k U_i$ such that $U_0$ does not contain any singularity of $\tF$. Note that $U_0$ is not connected and the restriction of $\tF$ and $\tF'$ on $U_0$ are regular. The strict conjugacy of the vanishing holonomy representations can be extended by path lifting method to the conjugacy
$$\Phi_0:(\tF,\tL)_{|U_0}\rightarrow (\tF',\tL')_{|U_0},$$
satisfying that the second eigenvalue function $v(\Phi_0)$ is identically $1$. We will show that on each intersection $V_i=U_i\cap U_0$, $\Phi_i$ and $\Phi_0$ coincide. Denote by
$$\Psi_i=\Phi^{-1}_{i| V_i}\circ \Phi_{0| V_i}:(\tF,\tL)_{|V_i}\rightarrow (\tF,\tL)_{|V_i}.$$
We claim that $v(\Psi_i)\equiv 1$ on $\DD\cap V_i$. Let $p,q$ in $V_i\cap\DD$. Denote by $l_p$ and $l_q$ the invariant curves of $\tL$ through $p$ and $q$ respectively. As the two maps $\Psi_{i|l_p}$ and $\Psi_{i|l_q}$ are conjugated by the holonomy transport, we have $v(\Psi_i)(p)=v(\Psi_i)(q)$. Consequently, $v(\Psi_i)$ is constant on $V_i$. Since $v(\Phi_0)\equiv 1$, it follows that $v(\Phi_i)$ is constant on $V_i\cap\DD$. Therefore, $v(\Phi_i)$ is constant on $U_i\cap\DD$. Moreover, at the singularity $p_i$, $D\Phi_i(p_i)$ has three eigenvectors corresponding to the directions of the divisor and the directions of invariant curves of $\tF$ and $\tL$ through $p_i$. Since $D\Phi_i(p_i)$ has also one eigenvalue $1$ corresponding to the directions of the divisor, we have $D\Phi_i(p_i)=\Id$. This gives  $v(\Phi_i)\equiv 1$ and consequently $v(\Psi_i)\equiv 1$.

Now at each point $p \in V_i\cap\DD$, the map $\Psi_{i|l_p}$ commutes with the holonomy of $\tF$ around $p_i$. Since the Camacho-Sad index $\lambda_i$ of $\tF$ at $p_i$ is not rational, Lemma \ref{lem10} below says that $\Psi_{i|l_p}=\Id$ and so $\Psi_i=\Id$. Hence we can glue all $\Phi_i$ together and the strict conjugacy we need is the projection of this diffeomorphism on $(\C^2,0)$ by $\sigma$.
\end{proof}

\begin{lemma}\label{lem10}
Let $h\in\Diff(\C,0)$ such that $h'(0)=\exp(2\pi i \lambda)$ where $\lambda\not\in\Q$. If $\psi\in\Diff(\C,0)$ satisfying $\psi'(0)=1$ and $\psi\circ h=h\circ \psi$ then $\psi=\Id$.
\end{lemma}
\begin{proof}
Since $\lambda\not\in\Q$,  there is a formal diffeomorphism $\phi$ such that  $\phi\circ h\circ\phi^{-1}(z)=\exp(2\pi i\lambda)z$. Denote by $\tilde{\psi}=\phi\circ \psi\circ\phi^{-1}$, then $\tilde{\psi}'(0)=1$ and  $\tilde{\psi}(\exp(2\pi i\lambda)z)=\exp(2\pi i\lambda)\tilde{\psi} $. The proof is reduced to show that $\tilde{\psi}=\Id$. Suppose that
$\tilde{\psi}(z)=z+\sum_{j=2}^{\infty}a_jz^j$. Then 
$$\tilde{\psi}(\exp(2\pi i\lambda)z)=\exp(2\pi i\lambda)z+\sum_{j=2}^{\infty}a_j\exp(2j\pi i\lambda)z^j,$$
and 
$$\exp(2\pi i\lambda)\tilde{\psi}(z)=\exp(2\pi i\lambda)z+\sum_{j=2}^{\infty}a_j\exp(2\pi i\lambda)z^j.$$
Since $\lambda\not\in\Q$, it forces $a_j=0$ for all $j\ge 2$. Hence $\tilde{\psi}=\Id$.
\end{proof}
\section{Finite determinacy}\label{sec4}
Let $S$ be a germ of curve at $p$ in a surface $X$. We denote by $\Sigma(S)$ the set of all germs of singular curves having the same desingularization map and having the same singularities as $S$ after desingularization. Here, the singularities of $S$ after desingularization are the singularities of the curve defined by the union of strict transform of $S$ and the exceptional divisor. If $S$ is smooth, we denote by $\m^n(S)$ the set of all holomorphic functions on $S$ whose vanishing orders at $p$ are  at least $n$. 
\begin{prop}\label{pro11}
Let $S$ be a germ of curve  in $(\C^2,0)$ and $S_1,\ldots,S_k$ be its irreducible components. Suppose that $\sigma:(\MM,\DD)\rightarrow(\C^2,0)$ is a finite composition of blowing-ups such that all the transformed curves $\sigma^*S_i=\tS_i$ are smooth. Then there exists a natural $N$ such that if $f_i\in\m^N(\tS_i)$, $i=1,\ldots,k$, then there exists $F\in\C\{x,y\}$ such that $F\circ\sigma_{|\tS_i}=f_i$.  Moreover, the same $N$ can be chosen for all elements in $\Sigma(S)$. 
\end{prop}
\begin{proof}
We first consider the statement when $S$ is irreducible. If $S$ is smooth then $\tS$ is diffeomorphic to $S$. So we can suppose that $S$ is singular. 
Denote by $p=\tS\cap\DD$. 
Choose a coordinate system $(x_p,y_p)$ in a neighborhood of $p$ such that  $\tS=\{y_p=0\}$ and $\DD=\{x_p=0\}$. Then $\sigma^{-1}$ is defined by 
$$x_p=\frac{\alpha(x,y)}{\beta(x,y)}\;\; \text{and}\;\; y_p=\frac{\mu(x,y)}{\nu(x,y)},$$
where $\alpha,\beta,\mu,\nu\in\C\{x,y\}$, $\mathrm{gcd}(\alpha,\beta)=1$ and $\mathrm{gcd}(\mu,\nu)=1$.
So we have
\begin{equation}
\frac{\alpha\circ\sigma(x_p,y_p)}{\beta\circ\sigma(x_p,y_p)}=x_p
\end{equation}
Therefore, there exist a natural $k$ and a holomorphic function $h$ such that 
\begin{equation}\label{equa11}
\alpha\circ\sigma(x_p,y_p)=x_p^{k+1}h(x_p,y_p),\,\,\beta\circ\sigma(x_p,y_p)=x_p^{k}h(x_p,y_p),
\end{equation} 
where $x_p\nmid h$. We claim that $h$ is a unit. Indeed, suppose $h(0,0)=0$ and denote by $\tilde{L}$ the curve $\{h(x_p,y_p)=0\}$. Let $L$ be a curve defined  in $(\C^2,0)$ such that $\sigma^*(L)=\tilde{L}$. Let $\{\bar{h}(x,y)=0\}$ be a reduced equation of $L$. By \eqref{equa11}, $\bar{h}|\alpha$ and $\bar{h}|\beta$. It contradicts $\mathrm{gcd}(\alpha,\beta)=1$.
Now denote by $u(x_p)=h(x_p,0)$ which is a unit, we have
\begin{equation}\label{11}
\alpha\circ\sigma(x_p,0)=u(x_p)x_p^{k+1},\,\beta\circ\sigma(x_p,0)=u(x_p)x_p^{k}.
\end{equation}
For each $m\ge (k-1)(k+1)$ there exists $j\in\{0,\ldots,k-1\}$ such that $k|(m-j(k+1))$. Thus
\begin{equation*}
m=ik+j(k+1),\, i,j\in\N.
\end{equation*}
So \eqref{11} implies that if a holomorphic function $f(x_p)$ satisfies $x_p^{(k-1)(k+1)}|f(x_p)$ then there exists a holomorphic function $F(x,y)$ such that $F\circ\sigma(x_p,0)=f(x_p)$. Consequently
\begin{equation}\label{equation13}
\m^{(k-1)(k+1)}(\tS)\subset\sigma^*\C\{x,y\}_{|\tilde{S}}.
\end{equation}
In the general case, suppose that $S_i$ is defined by $\{g_i=0\}$. If $f_i\in\m^{N}(\tS_i)$, $i=1,\ldots,k$, with $N$ big enough, there exist $F_i$, $i=1,\ldots,k$, such that $F_i\circ\sigma_{|\tS_i}=f_i$. We will find a holomorphic function $F$ such that $F_{|S_i}=F_{i|S_i}$ for all $i=1,\ldots,k$. This is reduced to show that there exists a natural $M$ such that the following morphism $\Theta$ is surjective
\begin{equation*}
\frac{(x,y)^M}{{(g_1)\cap\ldots\cap (g_k)\cap (x,y)^M}}\rightarrow\frac{(x,y)^M}{{(g_1)\cap (x,y)^M}}\oplus\ldots\oplus\frac{(x,y)^M}{{(g_k)\cap (x,y)^M}}.
\end{equation*}
Indeed, by Hilbert's Nullstellensatz, there exists a natural $M_1$ such that 
\begin{equation}\label{eq7}
(x,y)^{M_1}\subset (g_i,g_j)
\end{equation}
 for all $1\leq i < j \leq k$. We will show that for all $i=1,\ldots,k$, $j=0,\ldots,(k-1)M_1$ the elements $e_{ij}=(0,\ldots,\overline{x^jy^{(k-1)M_1-j}},\ldots,0)$, where $\overline{x^jy^{(k-1)M_1-j}}$ is in the $i^{th}$ position, are in $\mathrm{Im}\Theta$ and then $M$ can be chosen as $(k-1)M_1$. We decompose 
$$x^jy^{(k-1)M_1-j}=\prod_{\substack{l=1,\ldots,k\\ l\neq i}}x^{j_l}y^{M_1-j_l},$$
where $0\le j_l\le M_1$.  By \eqref{eq7}, there exist $a_{il},b_{il}\in\C\{x,y\}$ such that
 $a_{il}g_i+b_{il}g_l=x^{j_l}y^{M_1-j_l}$. This implies that 
 \begin{equation*}
 e_{ij}=\Theta\left(\prod_{\substack{l=1,\ldots,k\\ l\neq i}} \left(x^{j_l}y^{M_1-j_l}-a_{il}g_i\right)\right)\in\mathrm{Im}\Theta
 \end{equation*}
Now we will show that the same $N$ can be chosen for all elements of $\Sigma(S)$. In the case $S$ is irreducible, let $S'$ in $\Sigma(S)$ and $\{y_p=s(x_p)\}$ be the equation of $\sigma^*(S')=\tilde{S}'$. We also have
\begin{equation*}
\alpha\circ\sigma(x_p,s(x_p))=v(x_p)x_p^{k+1},\,\beta\circ\sigma(x_p,s(x_p))=v(x_p)x_p^{k},
\end{equation*}
where $v(x_p)=h(x_p,s(x_p))$ which is a unit. Consequently, \eqref{equation13} holds. In the general case, it is sufficient to show that the same $M_1$ in \eqref{eq7} can be chosen for all elements of $\Sigma(S)$. Let $M_{ij}$ be the smallest natural satisfying 
\begin{equation*}
(x,y)^{M_{ij}}\subset (g_i,g_j).
\end{equation*}
We claim that
$$M_{ij}\le \mathrm{I}(g_i,g_j)=\mathrm{dim}_\C\frac{\C\{x,y\}}{(g_i,g_j)}.$$ 
Indeed, there exists $x^{l}y^{M_{ij}-1-l}\not\in(g_i,g_j)$. Let $P_m$, $m=1,\ldots,M_{ij}$, be a sequence of monomials such that $P_1=1$, $P_{M_{ij}}=x^{l}y^{M_{ij}-1-l}$ and either $P_{m+1}=x\cdot P_m$ or $P_{m+1}=y\cdot P_m$. Since $P_m|P_{M_{ij}}$  we have $P_m\not\in(g_i,g_j)$ for all $m=1,\ldots, M_{ij}$. We will show that $\{P_1,\ldots, P_{M_{ij}}\}$ is independent in the vector space   $\frac{\C\{x,y\}}{(g_i,g_j)}$ over $\C$. Suppose that 
$$c_1P_1+\ldots+c_{M_{ij}}P_{M_{ij}}\in(g_i,g_j).$$
Suppose there exists $c_m\ne 0$. Let $m_0$ be the smallest natural such that $c_{m_0}\ne 0$. Then
$$c_{m_0}P_{m_0}+\ldots+c_{M_{ij}}P_{M_{ij}}=P_{m_0}(c_{m_0}+\ldots)\in(g_i,g_j).$$
This implies that $P_{m_0}$ in $(g_i,g_j)$ and it is a contradiction.

Now, it is well known that the intersection number $\mathrm{I}(g_i,g_j)$ is a topological invariant. It means that if two curves $\{g_i\cdot g_j=0\}$ and $\{g'_i\cdot g'_j=0\}$ are topologically conjugated then $\mathrm{I}(g_i,g_j)=\mathrm{I}(g'_i,g'_j)$. Consequently, $M_1$ can be chosen as $max_{1\le i<j\le k}\mathrm{I}(g_i,g_j)$ that doesn't depend on the elements of $\Sigma(S)$.
\end{proof}

Now, we will prove the finite determinacy property of the slidings of foliations.

\begin{proof}[Proof of Theorem \ref{thr2}]
Suppose that $\tilde{T}(\FF,\LL)=\cup T_i$, $\tilde{T}'(\FF',\LL')=\cup T'_i$  where $T_i$ and $T'_i$ are irreducible components of $T(\FF,\LL)$ and $T(\FF',\LL')$. Then the singularities of $\tF$ and $\tF'$ are $p_i=T_i\cap\DD=T'_i\cap\DD$. Denote by $h_{i\gamma}$ the holonomy of $\tF$ on $T_i$.

Now let $p_i$ be a singularity $\tF$. We first suppose that $p_i$ is not a corner. Denote by $D$ the irreducible component of $\DD$ through $p_i$. Because $\tF$ and $\tF'$ are strictly conjugated in a neighborhood of $p_i$, by Corollaries \ref{cor11} and \ref{cor3}, there exists a diffeomorphism $\psi_i$ in $\Diff(D,p_i)$ tangent to identity such that
\begin{equation}\label{equa13}
 \psi_i\circ g_{D,p_i}(\tF,\tL)=g_{D,p_i}(\tF',\tL')\circ\psi_i.
 \end{equation}
Let $\pi_D$ (resp., $\pi'_D$) be the projection from $T_i$ (resp., $T'_i$) to $D$ that follows the leaves of $\tL$ (resp., $\tL'$). Denote by
\begin{equation}\label{e13}
\phi_i=\pi_D^{-1}\circ\psi_i^{-1}\circ\pi_D.
\end{equation} 
Then $\phi_i\in\Diff(T_i,p_i)$. Since $J^{N}(S(\FF,\LL))=J^{N}(S(\FF',\LL'))$, $\phi_i$ is tangent to identity map at order at least $N$.

In the case $p_i$ is a corner, let $D$ be one of two irreducible components of $\DD$ through $p_i$ and define $\phi_i$ as above. We also have that $\phi_i$ is tangent to identity map at order at least $N$.
\begin{lemma}\label{lem17}
Suppose that there exists a diffeomorphism $\Phi$ such that the lifting $\sigma^*(\Phi)=\tilde{\Phi}$ satisfies
\begin{itemize}
\item{} $\tilde{\Phi}_{|\DD}=\Id$,
\item{} $\tilde{\Phi}_{|T_i}=\phi_i$,
\item{} $T(\FF,\Phi_*\LL)=T(\FF,\LL)$.
\end{itemize}
Then $\LL''=\Phi_*(\LL)$ satisfies $S(\FF,\LL'')=S(\FF',\LL')$.
\end{lemma}
\begin{proof}
Let $p_i$ be a singularity $\tF$. In the case $p_i$ is not a corner, we denote $D$, $\pi_D$, $\pi'_D$ as above. Let $\pi''_D$ be the projection following the leaves of $\tL''$ from $T_i$ to $D$, then 
$$\pi''_D=\pi_D\circ\phi_i^{-1}.$$
We have
\begin{align*}
g_{D,p_i}(\tF,\tL'')&=\pi''_D\circ h_{i\gamma} \circ\pi''^{-1}_D=\pi_D\circ\phi_i^{-1}\circ h_{i\gamma}\circ\phi_i\circ\pi_D^{-1}\\&= \psi_i\circ\pi_D\circ h_{i\gamma}\circ\pi_D^{-1}\circ\psi_i^{-1}= \psi_i\circ g_{D,p_i}(\tF,\tL)\circ\psi_i^{-1}\\&=g_{D,p_i}(\tF',\tL').
\end{align*} 
If $p_i$ is a corner, $p_i=D\cap D'$, we also have 
\begin{equation*}
g_{D,p_i}(\tF,\tL'')=g_{D,p_i}(\tF',\tL').
\end{equation*}
Since $\tilde{\Phi}_{|\DD}=\Id$ the Dulac maps of $\tL''$ and $\tL'$ in a neighborhood of $p_i$  are the same. So Remark \ref{re7} leads to
\begin{equation*}
g_{D',p_i}(\tF,\tL'')=g_{D',p_i}(\tF',\tL').
\end{equation*}
\end{proof}
Now we will prove the existence of $\Phi$ in Lemma \ref{lem17} for $N$ big enough. Suppose that $\FF$ and $\LL$ are respectively defined by
\begin{align*}
\omega&=a(x,y)dx+b(x,y)dy,\\
\omega_\LL&=c(x,y)dx+d(x,y)dy.
\end{align*}
Then the tangent curve $T=T(\FF,\LL)$ is defined by
$$q(x,y)=da-cb=0.$$
Let $X_q=\frac{\partial q}{\partial y}\ddx-\frac{\partial q}{\partial x}\ddy$ be a vector field tangent to $T$ and $\tilde{X}_q$ be its lifting by $\sigma$. By the implicit function theorem, if $N$ is big enough, there exists $f_i$ defined on $T_i$ such that 
$$\exp[f_i]\left(\restr{\tilde{X}_{q}}{T_i}\right)=\phi_i.$$
Using Proposition \ref{pro11}, by choosing  $N$ big enough, there exists $f\in\C\{x,y\}$ such that
$$\exp[f\circ\sigma]\restr{\tilde{X}_{q}}{T_i}=\phi_i.$$
For each $\Phi=(\Phi_1,\Phi_2)\in\Diff(\C^2,0)$, denote by $$<\Phi>=\frac{\omega_\LL\wedge\Phi^*\omega_\LL}{dx\wedge dy} =c\left(c\circ\Phi\frac{\partial\Phi_1}{\partial y}+d\circ\Phi\frac{\partial\Phi_2}{\partial y}\right)-d\left(c\circ\Phi\frac{\partial\Phi_1}{\partial x}+d\circ\Phi\frac{\partial\Phi_2}{\partial x}\right).$$
It is easy to see that $T(\FF,\Phi_*\LL)=T$ if and only if $q|<\Phi>$. For each holomorphic function $f$, we denote by 
$$\Phi_f= \exp[f]X_q.$$
Lemma \ref{lemma18} below says that there exists a holomorphic function $u$ such that $\Phi_{f-uq}$ satisfies Lemma \ref{lem17} for $N$ big enough. Moreover, by Proposition \ref{pro11}, we can chose $N$ that only depends on $\FF$.
\end{proof}
\begin{lemma}\label{lemma18}
If $N$ is big enough, for all $f$ in $(x,y)^N$ there exists a holomorphic function $u$ such that 
$q|<\Phi_{f-uq}>$.
\end{lemma}
\begin{proof}
We have
\begin{align*}
\restr{\frac{\partial}{\partial x}x\circ\Phi_{f-uq}}{\{q=0\}}&=\restr{\frac{\partial}{\partial x}\sum_{i=0}^{\infty}\frac{(f-uq)^i}{i!}\mathrm{ad}_{X_q}^i(x)}{\{q=0\}}\\
&=\restr{\frac{\partial}{\partial x}x\circ\Phi_f}{\{q=0\}}-\restr{u\cdot\frac{\partial q}{\partial x}\cdot\sum_{i=1}^{\infty}\frac{f^{i-1}}{(i-1)!}\mathrm{ad}^i_{X_q}(x)}{\{q=0\}}\\
&=\restr{\frac{\partial}{\partial x}x\circ\Phi_f}{\{q=0\}}-\restr{u\cdot\frac{\partial q}{\partial x}\cdot\frac{\partial q}{\partial y}\circ\Phi_f}{\{q=0\}}.
\end{align*}
Similarly,
\begin{align*}
\restr{\frac{\partial}{\partial y}x\circ\Phi_{f-uq}}{\{q=0\}}&=\restr{\frac{\partial}{\partial y}x\circ\Phi_f}{\{q=0\}}-\restr{u\cdot\frac{\partial q}{\partial y}\cdot\frac{\partial q}{\partial y}\circ\Phi_f}{\{q=0\}},\\
\restr{\frac{\partial}{\partial x}y\circ\Phi_{f-uq}}{\{q=0\}}&=\restr{\frac{\partial}{\partial x}y\circ\Phi_f}{\{q=0\}}+\restr{u\cdot\frac{\partial q}{\partial x}\cdot\frac{\partial q}{\partial x}\circ\Phi_f}{\{q=0\}},\\
\restr{\frac{\partial}{\partial y}y\circ\Phi_{f-uq}}{\{q=0\}}&=\restr{\frac{\partial}{\partial y}y\circ\Phi_f}{\{q=0\}}+\restr{u\cdot\frac{\partial q}{\partial y}\cdot\frac{\partial q}{\partial x}\circ\Phi_f}{\{q=0\}}.\\
\end{align*}
This implies that
\begin{align}
\restr{<\Phi_{f-uq}>}{\{q=0\}} &=\restr{<\Phi_f>}{\{q=0\}}\nonumber \\&-\restr{u\cdot\left(c\frac{\partial q}{\partial y}-d\frac{\partial q}{\partial x}\right)\cdot\left(\left(c\frac{\partial q}{\partial y}-d\frac{\partial q}{\partial x}\right)\circ\Phi_f\right)}{\{q=0\}}.\label{12}
\end{align}
Denote by $h=c\frac{\partial q}{\partial y}-d\frac{\partial q}{\partial x}$. Then $\{h=0\}$ is the tangent curve of $\LL$ and the foliation defined by the level sets of $q$. Since at each singularity $p_i$ of $\tF$, the irreducible component $T_i$ of $T$ is transverse to $\tL$, the irreducible components of the strict transform of $\{h=0\}$ at $p_i$ are also transverse to $T_i$. This implies that $(q,h)=1$ and the two curves $\{q\cdot h=0\}$ and $\{q\cdot (h\circ\Phi_f)=0\}$ are topologically conjugated.  
By Hilbert's Nullstellensatz and the proof of Proposition \ref{pro11} there exists a natural $M$ such that $(x,y)^M\subset(q,h)$ and  $(x,y)^M\subset(q,h\circ  \Phi_f)$. This implies that $(x,y)^{2M}\subset(q,h\cdot(h\circ \Phi_f))$. So if $<\Phi_{f}>\in(x,y)^{2M}$, by \eqref{12} we can choose $u\in\C\{x,y\}$ such that $q|<\Phi_{f-uq}>$.
\end{proof}

\begin{rema}
If we replace the condition ``$\tF$ and $\tF'$ are locally strictly analytically conjugated" in Theorem \ref{thr2} by the condition ``$\FF$ and $\FF'$ are in $\MM$'' then  the conclusion in Theorem \ref{thr2} becomes: ``For all natural $M\ge N$ there exists $\LL''_M$ such that $J^M(S(\FF,\LL''_M))=J^M(S(\FF',\LL'))$''. Indeed, in that case, because the Camacho-Sad indices are not rational, $\tF$ and $\tF'$ are locally formally conjugated. So we can choose $\psi$ in \eqref{equa13} such that 
$$J^{M}(\psi_i\circ g_{D,p_i}(\tF,\tL)\circ\psi_i^{-1})=J^{M}(g_{D,p_i}(\tF',\tL')).$$
\end{rema}

\begin{coro}\label{cor17}
Let $\FF\in\mathcal{M}$ defined by a $1$-form $\omega$ then there exists a natural $N$ such that if $\FF'\in\mathcal{M}$ is defined by a $1$-form $\omega'$ satisfying that $J^N\omega=J^N\omega'$ and the vanishing holonomy representations of $\FF$ and $\FF'$ are strictly analytically conjugated, then $\FF$ and $\FF'$ are strictly analytically conjugated.
\end{coro}
\begin{proof}
Let $\LL\in\mathcal{R}_0$ then $J^{m(N)}S(\FF,\LL)=J^{m(N)}S(\FF',\LL)$ where $m(M)$ is an increasing function on $N$ and $m(N)\rightarrow\infty$ when $N\rightarrow\infty$. By Theorem \ref{thr2} if $N$ is big enough there exists $\LL''\in\mathcal{R}_0$ such that $S(\FF,\LL'')=S(\FF',\LL)$. By Theorem \ref{thr1}, $\FF$ and $\FF'$ are strictly analytically conjugated.
\end{proof}
\begin{rema} This Corollary is consistent with the result of J.F. Mattei in \cite{Mat1} which says that the dimension of moduli space of the equisingular unfolding of a foliation is finite. Note that the vanishing holonomy representations of two foliations that are jointed by a unfolding are conjugated but the converse is not true.
\end{rema}

\begin{coro}\label{cor19}
Let $\LL$ be a $\sigma$-absolutely dicritical foliation defined by $1$-form $\omega$. There exists a natural $N$ such that if $\LL'$ is a $\sigma$-absolutely dicritical foliation defined by   $\omega'$ satisfying $J^N\omega=J^N\omega'$ and the Dulac maps of  $\tL$ and $\tL'$ are the same then $\LL$ and $\LL'$ are strictly analytically conjugated.
\end{coro}
\begin{proof}
Suppose that $\DD=\cup_{i=1,\ldots, k}D_i$ where $D_i$ is an irreducible component of $\DD$.  We take a pair of irreducible functions $f_i$ and $g_i$ for each $i=1,\ldots, k$, such that the curve $C_i=\{f_i=0\}$ and $C'_i=\{g_i=0\}$ satisfy the following properties:
\begin{itemize}
\item[1.] The strict transforms $\tilde{C_i}$ and $\tilde{C'_i}$ cut $D_i$ at two different points $p_i$, $q_i$, respectively, such that none of them is a corner.
\item[2.] $\tilde{C_i}$, $\tilde{C'_i}$ are smooth and transverse to the invariant curve of $\tL$ through $p_i$, $q_i$ respectively. 
\end{itemize} 
Because $[\C:\Q]$ is an infinite field extension, there exists  $(\lambda_1,\lambda_2,\ldots,\lambda_k)\in\C^k$ such that 
$$\sum_{i=1}^k c_i\lambda_i\not\in\Q,\;\forall (c_1,\ldots,c_k)\in\Q^k\setminus\{(0,\ldots,0)\}.$$
Now, let us consider the non-dicritical foliation $\FF$ defined by the $1$-form
$$\omega_0=\prod_{i=1}^k(f_ig_i)\cdot \left(\sum_{i=1}^k\left(\lambda_i\frac{df_i}{f_i}+\frac{dg_i}{g_i}\right)\right).$$
Then $\FF$ admits $\sigma$ as its desingularization map and the singularities of the strict transform $\tF$ are the corners of $\DD$ and $p_i,q_i$, $i=1,\ldots,k$.  We claim that at each singularity, the Camacho-Sad index of $\tF$ is not rational. Indeed, denote by $m_{ij}$ the multiplicity of $f_i\circ\sigma$ and $g_i\circ\sigma$ on $D_j$. At the corner $p_{ij}=D_i\cap D_j$, we take  coordinates $(x,y)$ such that $D_i=\{x=0\}, D_j=\{y=0\}$.  In this coordinate system, we can write $\sigma^*\omega_0$ as
\begin{equation*}
\sigma^*\omega_0=u(x,y)x^{2\sum_{l=1}^{k} m_{li}}y^{2\sum_{l=1}^{k} m_{lj}}\sum_{l=1}^k\left((\lambda_l+1)m_{li}\frac{dx}{x}+(\lambda_l+1)m_{lj}\frac{dy}{y}+ \alpha_{ij}\right),
\end{equation*}
where $u(x,y)$ is a unit and $\alpha_{ij}$ is a holomorphic form. So the Camacho-Sad index of $\tF$ at $p_{ij}$ is
\begin{equation*}
\mathrm{I}(p_{ij})=\frac{\sum_{l=1}^k(\lambda_l+1)m_{lj}}{\sum_{l=1}^k(\lambda_l+1)m_{li}}\not\in\Q.
\end{equation*}
Similarly, the Camacho-Sad indices of $\tF$ at $p_i$ and $q_i$, respectively, are
\begin{equation*}
\mathrm{I}(p_i)=\frac{\sum_{l=1}^k(\lambda_l+1)m_{li}}{\sum_{l=1}^k \lambda_l}\not\in\Q,\; \mathrm{I}(q_i)=\frac{\sum_{l=1}^k(\lambda_l+1)m_{li}}{k}\not\in\Q.
\end{equation*} 
Now if $J^N\omega=J^N\omega'$ then $J^{m(N)}S(\FF,\LL)=J^{m(N)}S(\FF,\LL')$ where $m(N)$ is an increasing function on $N$ and $m(N)\rightarrow\infty$ when $N\rightarrow\infty$. Moreover if $N$ is big enough the invariant curves of $\tL$ and $\tL'$ through the singularities of $\tF$ are tangent. By using Theorem \ref{thr2} for $\FF'=\FF$, there exists a foliation $\LL''$ strictly conjugated with $\LL$ such that the two couples $(\FF,\LL'')$ and $(\FF,\LL')$ are strictly conjugated. Consequently, $\LL$ and $\LL'$ are strictly conjugated. 
\end{proof}

\section*{Acknowledgements}
This paper is based on the main part of my doctoral thesis at the Institute of Mathematics of Toulouse, France. I would like to thank my advisors,Yohann Genzmer and  Emmanuel Paul, for having
guided me so well over all these years, Jean-François Mattei for useful discussions and suggestions about the problem.

\bigskip

\end{document}